\documentclass[10pt]{article}
\usepackage[english]{babel}
\usepackage{amsmath}
\usepackage{amsthm}
\usepackage{amssymb}

\newtheorem{theorem}{Theorem}[section]
\newtheorem{proposition}[theorem]{Proposition}

\newtheorem{lemma}[theorem]{Lemma}
\newtheorem{remark}[theorem]{Remark}

\newcommand{\mL}{{\cal L}}

\newcommand{\R}{\mathbb{R}}

\newcommand{\beq}{\begin{equation}}
\newcommand{\eeq}{\end{equation}}

\newcommand{\N}{\mathbb{N}}
\newcommand{\Z}{{\mathbb{Z}}}
\newcommand{\A}{{\cal A}}
\newcommand{\mB}{{\mathcal{B}}}

\title{\bf Parabolic problems with general Wentzell boundary conditions and diffusion on the boundary}

\author{Davide Guidetti\\ \\
Dipartimento di matematica, \\
Piazza di Porta S. Donato 5, \\
40126 Bologna (Italy) \\
e-mail: davide.guidetti@unibo.it
}
\date{}
\begin{document}
\maketitle

\begin{abstract}
We show a result of maximal regularity in spaces of H\"older continuous function, concerning linear parabolic systems, with dynamic or Wentzell boundary conditions, with an elliptic diffusion term on the boundary. 
\end{abstract}


\medskip


\medskip

\section{Introduction} 

In this paper we want to study second order parabolic systems in the forms

\begin{equation}\label{eq1.1}
\left\{\begin{array}{ll}
D_tu(t,x) = A(t,x, D_x) u(t,x) + f(t,x), & (t,x) \in [0, T] \times \Omega, \\ \\
D_tu(t,x') + B(t,x', D_x) u(t,x') - L(t)(u(t,\cdot)_{|\Gamma})(x')  = h(t,x'), & (t,x') \in [0, T] \times \partial \Omega, \\ \\
u(0,x) = u_0(x), & x \in \Omega. 
\end{array}
\right. 
\end{equation}
and
\begin{equation}\label{eq1.1B}
\left\{\begin{array}{ll}
D_tu(t,x) = A(t,x, D_x) u(t,x) + f(t,x), & (t,x) \in [0, T] \times \Omega, \\ \\
A(t,x', D_x) u(t,x') + B(t,x', D_x) u(t,x') - L(t)(u(t,\cdot)_{|\Gamma})(x')  = h(t,x'), & (t,x') \in [0, T] \times \partial \Omega, \\ \\
u(0,x) = u_0(x), & x \in \Omega. 
\end{array}
\right. 
\end{equation}
Here for every $t \in [0, T]$, $A(t,x,D_x)$ is a second order linear strongly elliptic operator in the open, bounded subset $\Omega$ of $\R^n$,  $L(t)$ is a second order linear strongly elliptic tangential operator in $\partial \Omega$, $B(t,x', D_x)$ is a first order (not necessarily tangential) operator in $\partial \Omega$. It is clear that, at least formally, (\ref{eq1.1}) and (\ref{eq1.1B}) are strictly related. 

A large amount of papers has been devoted to parabolic problems with dynamic and Wentzell boundary conditions in the form (\ref{eq1.1})-(\ref{eq1.1B}) in the case that the summand $L(t)(u(t,\cdot)_{|\Gamma})$ does not appear. We refer to the bibliographies in \cite{Gu3} and \cite{Gu2}. In our knowledge, a problem in the form (\ref{eq1.1}) was introduced for the first time in \cite{Ru1}, in the particular case that $A(t,x,D_x) = \alpha(x) \Delta_x$, with $\alpha$ positively valued,  $B(t,x', D_x) =  b(x') D_{\nu}$, with $\nu$ unit normal vector to $\partial \Omega$, pointing outside $\Omega$, $L(t) = a(x) \Delta_{LB}u$, where we indicate with $\Delta_{LB}$ the Laplace-Beltrami operator.  \cite{Ru1} contains a physical interpretation of the problem: briefly, a heat equation with a heat source on the boundary, that depends on the heat flow along the boundary, the heat flux across the boundary and the temperature at the boundary. 

The first paper where a problem in the form (\ref{eq1.1B}) is really studied seems to be \cite{FaGoGoObRo1}. In it it was considered the system
\begin{equation}\label{eq1.3}
\left\{\begin{array}{ll}
D_t u(t,x) = Au(t,x) = \nabla \cdot (a(x) \nabla u) (t,x), & (t,x) \in (0, T) \times \Omega, \\ \\
Au(t,x') + \beta(x') D_{\nu_A} u(t,x') + \gamma(x') -q \Delta_{LB}u(t,x') = 0, & (t,x') \in (0, T) \times \partial \Omega, \\ \\
u(0,x) = u_0(x),
\end{array}
\right. 
\end{equation}
with , $A$ strongly elliptic, $\beta(x') > 0$ in $\partial \Omega$, $D_{\nu_A}$ conormal derivative, $q \in [0, \infty)$. It is proved that, if $1 \leq p \leq \infty$ the closure a suitable realisation of the problem in the space $L^p(\Omega \times \partial \Omega)$ ($1 \leq p \leq \infty$), gives rise to an analytic semigroup (not strongly continuous if $p = \infty)$. The continuous dependence on the coefficients had already been considered in \cite{CoFaGoGoRo1}. 

In \cite{ClGoGoRo1} the authors generalised some of the results in \cite{FaGoGoObRo1}, considered also the case that the first equation in (\ref{eq1.3}) is the telegraph equation (with two initial conditions) and studied the asymptotic behaviour of solutions. 

In \cite{Wa3} the author considered the case of a domain  $\Omega$  with merely  Lipschitz boundary, with a strongly elliptic operator $A$ (independent of $t$). It was shown that a realisation of $A$ with the general boundary condition  $(Au)_{|\partial \Omega} - \gamma \Delta_{LB} u + D_{\nu_A} u + \beta u = g$ in $\partial \Omega$ generates a strongly continuous compact semigroup in $C(\overline \Omega)$. Semilinear problems were studied in \cite{Wa1} and \cite{Wa2}. 

Finally, in the paper \cite{VaVi1} the authors treated (\ref{eq1.1}) in the particular case  $A(t,x,D_x) = \Delta_x$, $f \equiv 0$, $h \equiv 0$, $L(t) =  l \Delta_{LB}$ with $l > 0$ and $B(t,x',D_x) = k D_{\nu}$, with $k$ which may be negative (in contrast with the previously quoted literature). They showed that, if the initial datum $u_0$ is in $H^1(\Omega)$ and $u_{0|\partial \Omega} \in H^1(\Gamma)$, (\ref{eq1.1}) has a unique solution $u$ in $C([0, \infty); H^1(\Omega)) \cap C^1((0, \infty); H^1(\Omega)) \cap C((0, \infty); H^3(\Omega))$, with $u_{|\Gamma}$ in $C([0, \infty); H^1(\Gamma)) \cap C^1((0, \infty); H^1(\Gamma)) \cap C((0, \infty); H^3(\Gamma))$. 

The main aim of this paper is to show that, in a suitable functional setting, the role of the operator $B(t,x',D_x)$ in (\ref{eq1.1}) and (\ref{eq1.1B}) is minor, in the sense that these equations can be treated as perturbations of
 the corresponding problems with $B(t,x',D_x) \equiv 0$. In fact, we shall see that  $B(t,x',D_x)$ may be, apart some limitations on the regularity of its coefficients, an arbitrary first order linear differential operator. Moreover we shall consider 
 problems with coefficients depending on $t$ and we shall obtain results of maximal regularity, that is, results establishing the existence of linear and topological isomorphisms between classes of data and classes of solutions. Following the 
 lines of \cite{Gu4} and \cite{Gu2},  we shall work in spaces of H\"older continuous functions. 
 
Now we are going to state our main results. We begin by introducing the following assumptions: 

\medskip

{\it (A1) $\Omega$ is an open, bounded subset of $\R^n$ ($n \in \N$, $n \geq 3$), lying on one side of its topological boundary $\Gamma$, which is a compact submanifold of $\R^n$ of class $C^{2 + \beta}$, for some $\beta \in (0, 1)$. 

\medskip

(A2) $A(t,x,D_x) = \sum_{|\alpha| \leq 2} a_\alpha(t,x) D_x^\alpha$, with $a_\alpha \in C^{\beta/2,\beta} ((0, T) \times  \Omega)$; in case $|\alpha| = 2$, $a_\alpha$ is real valued and there exists $\nu > 0$ such that, $\forall (t,x)  \in [0, T] \times \overline \Omega$, 
$\forall \xi \in \R^n$, $\sum_{|\alpha| = 2} a_\alpha(t,x) \xi^\alpha \geq \nu |\xi|^2$ (we shall briefly say that $A(x,D_x)$ is strongly elliptic). 

\medskip

(A3) $B(t,x', D_x) = \sum_{|\alpha| \leq 1} b_\alpha(t,x') D_x^\alpha$, with $b_\alpha \in C^{\beta/2,\beta} ((0, T) \times \Gamma)$. 

\medskip

(A4) $\forall t \in [0, T]$ $L(t)$ is a second order, partial differential operator in $\Gamma$. More precisely: for every local chart $(U, \Phi)$, with $U$ open in $\Gamma$ and $\Phi$ $C^{2+\beta}-$ diffeomorphism between $U$ and $\Phi(U)$, with $\Phi(U)$ open in $\R^{n-1}$,  $\forall v \in C_0^{2+\beta}(\Gamma)$ with compact support in $U$, 
$$
L (t)v(x')  = \sum_{|\alpha| \leq 2} l_{\alpha, \Phi}(t,x') D_y^\alpha (v \circ \Phi^{-1})(\Phi(x'));
$$
we suppose, moreover, that,  if $|\alpha| = 2$, $l_{\alpha, \Phi}$ is real valued, for every open subset $V$ of $U$, with $\overline V \subset \subset U$,  $l_{\alpha, \Phi |V} \in C^{\beta/2,\beta}((0, T) \times V)$,    and  there exists $\nu(V)$ positive such that, $\forall (t,x') \in V$,   
$\forall \eta \in \R^{n-1}$, $\sum_{|\alpha| = 2} l_{\alpha, \Phi} (t,x') \eta^\alpha \geq \nu(V)  |\eta|^2$.

}

\medskip

We want to prove the following

\begin{theorem}\label{th1.1}
Suppose that (A1)-(A4) are fulfilled. Then the following conditions are necessary and sufficient in order that (\ref{eq1.1}) have a unique solution $u$ belonging to $C^{1+\beta/2, 2+\beta}((0, T) \times \Omega)$: 

(a) $f \in C^{\beta/2, \beta}((0, T) \times \Omega)$;

(b) $h \in C^{\beta/2, \beta}((0, T) \times \Gamma)$;

(c) $u_0 \in C^{2+\beta}(\Omega)$; 

(d) $A(0,x',D_x) u_0(x') + f(0,x') = - B(0,x', D_x) u_0(x') + L (0)(u_{0|\Gamma})(x') + h(0,x'), \quad \forall x' \in \Gamma.$

\end{theorem}

\begin{theorem}\label{th1.2}
Suppose that (A1)-(A4) are fulfilled. 
Then the following conditions are necessary and sufficient in order that (\ref{eq1.1B}) have a unique solution in $C^{1+\beta/2, 2+\beta}((0, T) \times \Omega)$: 

(a) $f \in C^{\beta/2, \beta}((0, T) \times \Omega)$;

(b) $h \in C^{\beta/2, \beta}((0, T) \times \Gamma)$;

(c) $u_0 \in C^{2+\beta}(\Omega)$; 

(d) $A(0,x',D_x) u_0(x') +  B(0,x', D_x)  - L (0)(u_{0|\Gamma})(x') = h(0,x'), \quad \forall x' \in \Gamma$.

\end{theorem}

Now we are going to describe the organisation of the paper. We begin by considering in Section \ref{se2} the parabolic problem 
\begin{equation}\label{eq2.2}
\left\{\begin{array}{ll}
D_t g(t,x') = Lg(t,x') + h(t,x'), & (t,x') \in (0, T) \times \Gamma, \\ \\
g(0,x') = g_0(x'), & x' \in \Gamma,
\end{array}
\right.
\end{equation}
with $L$ strongly elliptic in $\Gamma$. We do not impose the variational form of $L$. We show that the operator $\mL$, defined as 
\begin{equation}
\left\{\begin{array}{l}
D(\mathcal L) = \{g \in \cap_{1 \leq p < \infty} W^{2,p}(\Gamma) : Lg \in C(\Gamma)\}, \\ \\
\mathcal L g = Lg. 
\end{array}
\right.
\end{equation}
is the infinitesimal generator of an analytic semigroup in $C(\Gamma)$. This can be easily obtained, by local charts methods, employing well known analogous results in $\R^n$ (see \cite{Lu1}, Chapter 3). Employing
maximal regularity techniques in spaces of continuous and H\"older continuous functions (see again \cite{Lu1}), we determine  in Proposition \ref{pr2.3} necessary and sufficient conditions (analogous to well known conditions in $\R^{n-1}$), in order that (\ref{eq2.2}) have a unique solution in $C^{1+\beta/2,2+\beta}((0, T) \times \Gamma)$ (we shall recall the definition of these classes in the following). This first step is admittedly 
simply, but we were not able to find it in literature. 

In Section \ref{se3} we study systems (\ref{eq1.1}) and (\ref{eq1.1B}). Employing the results of Section \ref{se2} we begin by determining in Theorem \ref{th3.1}  necessary and sufficient conditions such that system (\ref{eq1.1}) have a unique solution in $C^{1+\beta/2,2+\beta}((0, T) \times \Omega)$ in the particular case $A(t,x,D_x) = A(x,D_x)$, $B(t,x',D_x) = 0$, $L(t) = L$ (independent of $t$). Finally, we obtain Theorem \ref{th1.1} from this particular case, by perturbation arguments. Theorem \ref{th1.2} is a simple consequence of Theorem \ref{th1.1}. 

We conclude this preliminary section by specifying some notations and by recalling some facts that we shall use. 

$C$, $C_0$, $C_1, \dots$ will indicate positive constants that we are not interested to precise and may be different from time to time. We shall write $C(\alpha, \beta, \dots)$ to indicate that the constant depends on $\alpha, \beta, \dots$. 

If $L$ is a tangential differential operator in the boundary $\Gamma$ and $u$ is defined in $[0, T] \times \overline \Omega$, we shall write $Lu(t,x')$ ($x' \in \Gamma$),  instead of $L(u(t,\cdot)_{|\Gamma})(x')$. If $X$ and $Y$ are Banach spaces, we shall indicate with $\mL(X,Y)$ the Banach space of linear, bounded operator from $X$ to $Y$. In case $X = Y$, we shall write $\mL(X)$. 

Let $X_0, X. X_1$ be Banach spaces, with $X_1 \hookrightarrow X \hookrightarrow X_0$, and let $\theta \in (0, 1)$. We shall write $X \in J_\theta (X_0,X_1)$ if there exists $C$ positive, such that, $\forall x \in X_1$, 
$$
\|x\|_X \leq C \|x\|_{X_0}^{1-\theta} \|x\|_{X_1}^{\theta}. 
$$

If $X_0$ and $X_1$ are complex, compatible Banach spaces, $\theta \in (0, 1)$, $p \in [1, \infty]$, we shall use the standard notation $(X_0,X_1)_{\theta,p}$ to indicate the corresponding real interpolation space. 

Let $\Omega$ be an open subset of $\R^n$. We shall indicate with  $B(\Omega)$ and $C(\Omega)$ the spaces of (respectively) complex valued, bounded and complex valued, uniformly continuous and bounded functions with domain $\Omega$. If $f \in C(\Omega)$, it is continuously extensible to its topological closure $\overline \Omega$. We shall identify $f$ with this extension. If $m \in \N$, we indicate with $C^m(\Omega)$ the class of functions $f$ in $C(\Omega)$, whose derivatives $D^\alpha f$, with order $|\alpha| \leq m$, belong to $C(\Omega)$. We shall equip these spaces with their natural norms:
$$
\|f\|_{C^m(\Omega)} = \max\{\|D^\alpha f\|_{C(\Omega)}: |\alpha| \leq m\},
$$
with $\|f\|_{C(\Omega)} := \sup_{x \in \Omega} |f(x)|$. If $\beta \in (0, 1)$, we set
$$[f]_{C^\beta(\Omega)}: = \sup_{x,y \in \Omega, x \neq y} |x-y|^{-\beta} |f(x) - f(y)|\}$$ 
and, if $m \in \N_0$, 
\begin{equation}
\|f\|_{C^{m+\beta}(\Omega)}:= \max\{\| f\|_{C^m(\Omega)}, \max\{ [D^\alpha f]_{C^\beta(\Omega)}Ê: |\alpha| = m\}\};
\end{equation}
of course, $C^{m+\beta}(\Omega) = \{f \in C^m(\Omega): \|f\|_{C^{m+\beta}(\Omega)} < \infty\}$. The previous definition can be extended in an obvious way to functions $f : \Omega \to X$, with $X$ Banach space. 

Let $0 \leq \beta_0 < \beta < \beta_1$. If there exists a common bounded extension operator from $C^\xi(\Omega)$ to $C^\xi(\R^n)$ $\forall \xi \in [\beta_0, \beta_1]$, then $C^\beta(\Omega) \in J_{\frac{\beta - \beta_0}{\beta_1 - \beta_0}}(C^{\beta_0}(\Omega), C^{\beta_1}(\Omega))$. This is a consequence of the embedding 
$$
(C^{\beta_0}(\Omega), C^{\beta_1}(\Omega))_{\frac{\beta - \beta_0}{\beta_1 - \beta_0}} = B^\beta_{\infty,1}(\Omega) \hookrightarrow C^\beta(\Omega). 
$$
(se Theorem 4 in \cite{Gu4} and the indicated references). 

If $I$ is an open interval in $\R$, $\Omega$ is an open subset of $\R^n$ and $\alpha, \beta$ are nonnegative, we set
\begin{equation}
C^{\alpha,\beta}(I \times \Omega):= C^\alpha(I; C(\Omega)) \cap C^\beta(\Omega; C(I)). 
\end{equation}
This is a Banach space, with the norm
\begin{equation}
\|f\|_{C^{\alpha,\beta}(I \times \Omega)}:= \max\{\|f\|_{C^{\alpha}(I; \Omega)}, \|f\|_{C^{\beta}(\Omega; C(I))}\}.
\end{equation}

The following facts hold (see \cite{Gu4}), Lemma 1: 

\begin{lemma}\label{le2.7B} (I) $C^\beta(\Omega; C(I)) \subseteq  B(I; C^\beta(\Omega))$. 

(II) Suppose $\alpha, \beta \geq 0$ with $\beta \not \in \Z$ and $\Omega$ such that there exists a common linear bounded extension operator, mapping $C(\Omega)$ into $C(\R^n)$ and
$C^\beta(\Omega)$ into $C^\beta(\R^n)$. Then, 
$$C^{\alpha,\beta}(I \times \Omega) = C^{\alpha}(I; C(\Omega)) \cap B(I; C^\beta(\Omega)). $$
Let $\beta \in (0,1)$ and suppose that   there exists a common linear bounded extension operator mapping $C^\gamma(\Omega)$ into $C^\gamma(\R^n)$, $\forall \gamma \in [0, 2+\beta]$. 
Then

(III) $C^{1+\beta/2, 2+\beta}(I \times \Omega) = C^{1+\beta/2}(I; C(\Omega)) \cap B(I; C^{2+\beta}(\Omega))$; 

(IV) if $f \in C^{1+\beta/2, 2+\beta}(I \times \Omega)$, $D_t f \in B(I; C^\beta(\Omega))$; 

(V) $C^{1+\beta/2, 2+\beta}(I \times \Omega) \subseteq C^{\frac{1+\beta}{2}}(I; C^1(\Omega)) \cap C^{\beta/2}(I; C^2(\Omega))$.

\end{lemma}

The previous definitions and results can be extended (by local charts) to functions $f : \Gamma \to X$, with $\Gamma$ suitably smooth differentiable manifold. In Section ref{se2} we shall also deal with the Besov space $B^2_{\infty,\infty}(\Gamma)$. This space can be defined by local charts, employing the following definition of the space $B^2_{\infty,\infty}(\R^{n-1})$: 
$$
B^2_{\infty,\infty}(\R^{n-1}) := \{f \in C^1(\R^{n-1}) : sup_{h \neq 0} |h|^{-1} |\nabla f(x+ h) - 2 \nabla f(x) + \nabla f(x-h)| < \infty\}. 
$$
It is known (see COMPLETARE ) that $B^2_{\infty,\infty}(\R^{n-1})$ properly contains the space $W^{2,\infty}(\R^n)$ of elements of $C^1(\R^n)$ with Lipschitz continuous first order derivatives. On the other hand, $B^2_{\infty,\infty}(\R^{n-1}) \subseteq C^{2-\epsilon}(\R^n)$ $\forall \epsilon \in (0, 2]$. We shall consider also spaces $W^{2,p}(\Gamma)$, just in the case of $\Gamma$ compact and of class $C^{2+\beta}$ ($\beta > 0$), which can be again defined by local charts. In Section \ref{se2} we shall employ Besov spaces $B^\alpha_{\infty,\infty}(\Gamma)$, with $\alpha \in [0, 2]$. We shall need the following facts, which can be easily deduced from analogous statements in $\R^{n-1}$
(see \cite{Gu5}): 

\begin{lemma} \label{le1.4}
Let $\Gamma$ be a compact submanifold of $\R^n$, of class $C^{2+\beta}$, for some $\beta > 0$. Then

(I)  if $\theta  \in [0, 2]$, $C^\theta(\Gamma) \subseteq B^\theta_{\infty,\infty}(\Gamma)$; 

(II) in case $\theta \not \in \Z$, $C^\theta(\Gamma) = B^\theta_{\infty,\infty}(\Gamma)$; 

(III) $\forall \theta \in (0, 1) \setminus \{\frac{1}{2}\}$, 
$$(C(\Gamma), C^2(\Gamma))_{\theta,\infty} = (B^0_{\infty,\infty}(\Gamma), B^2_{\infty,\infty}(\Gamma))_{\theta,\infty} = B^{2\theta}_{\infty,\infty}(\Gamma) = C^{2\theta}(\Gamma). $$
\end{lemma}

(a)

(b) in any case,

We shall employ the following version of the continuation principle: 

\begin{proposition}\label{pr1.20B}
Let $X, Y$ be Banach spaces and $L \in C([0, 1]; \mL(X,Y))$. Assume the following:

(a) there exists $M \in \R^+$, such that, $\forall x \in X$, $\forall \epsilon \in [0, 1]$,
$$
\|x\|_X \leq M \|L(\epsilon) x\|_Y; 
$$
(b) $L(0)$ is onto $Y$.

Then, $\forall \epsilon \in [0, 1]$ $L(\epsilon)$ is a linear and topological isomorphism between $X$ and $Y$. 
\end{proposition}

\section{Parabolic problems in $\Gamma$}\label{se2}

\setcounter{equation}{0}

In this section, we study the parabolic system (\ref{eq2.2}). We introduce the following as

\medskip

{\it (A4'). $L$ is a second order, partial differential operator in $\Gamma$. More precisely: for every local chart $(U, \Phi)$, with $U$ open in $\Gamma$, $\forall v \in C_0^{2+\beta}(\Gamma)$ with compact support in $U$, 
$$
L v(x')  = \sum_{|\alpha| \leq 2} l_{\alpha, \Phi}(x') D_y^\alpha (v \circ \Phi^{-1})(\Phi(x'));
$$
we suppose, moreover, if $|\alpha| = 2$, $l_{\alpha, \Phi}$ is real valued, for every open subset $V$ of $U$, with $\overline V \subset \subset U$,  $l_{\alpha, \Phi |V} \in C^{\beta/2,\beta}((0, T) \times V)$,    and, there exists $\nu(V)$ positive such that, $\forall x' \in V$,   
$\forall \eta \in \R^{n-1}$, $\sum_{|\alpha| = 2} l_{\alpha, \Phi} (x') \eta^\alpha \geq \nu(V)  |\eta^2$.

}

\medskip

 We begin by considering the elliptic system depending on the parameter $\lambda$
\begin{equation}\label{eq2.2A}
\lambda g(x') - Lg(x') = h(x'), \quad x' \in \Gamma. 
\end{equation}
We shall prove the following

\begin{theorem}\label{th2.1}
Suppose that (A1) and (A4) hold. Then:

(I) for every $\phi_0 \in [0, \pi)$ there exists $R(\phi_0) > 0$ such that, if $|Arg(\lambda)| \leq \phi_0$ and $|\lambda| \geq R(\phi_0)$, (\ref{eq2.2}) has a unique solution $g$ in $\bigcap_{1 \leq p < \infty} W^{2,p}(\Gamma)$; 

(II) $g$ belongs also to the Besov space $B^2_{\infty,\infty}(\Gamma)$; moreover, for every $\gamma \in [0, 2)$ there exists $C > 0$, depending on $\phi_0$ and $\gamma$, such that
$$
\|g\|_{C^\gamma(\Gamma)} \leq C |\lambda|^{-1+ \gamma/2} \|h\|_{C(\Gamma)}; 
$$

(III) if $h \in C^\beta(\Gamma)$, $g \in C^{2+\beta}(\Gamma)$. 
\end{theorem}

\begin{proof} We take an arbitrary $x^0 \in \Gamma$ and consider a local chart $(U, \Phi)$ around $x^0$, with $U$ open subset of $\Gamma$ and $\Phi$ diffeomorphism between $U$ and $\Phi(U)$, open subset in $\R^{n-1}$. We introduce in $\Phi(U)$
the strongly elliptic operator $B^\sharp$, 
$$
L^\sharp v(y) := L(v \circ \Phi)(\Phi^{-1}(y)), \quad y \in \Phi(U). 
$$
By shrinking $U$ (if necessary), we may assume that the coefficients of $L^\sharp$ are in $C^\beta(\Phi(U))$ and are extensible to elements $l_\beta$ in $C^\beta(\R^n)$, in such a way that the operator which we continue to call $L^\sharp = \sum_{|\alpha| \leq 2}  l_\beta (y) D_y^\beta$ is strongly elliptic in $\R^n$. Now we consider the problem
\begin{equation}\label{eq2.3}
\lambda v(y) - L^\sharp v(y) = k(y), \quad y \in \R^{n-1}. 
\end{equation}
Then, (see Chapter 3 in \cite{Lu1}), for every $\phi_0 \in [0, \pi)$ there exist  $R_1(\phi_0) > 0$ such that, if $|Arg(\lambda)| \leq \phi_0$ and $|\lambda| \geq R_1(\phi_0)$, (\ref{eq2.3}) has a unique solution $v$ in $\bigcap_{1 \leq p < \infty} W^{2,p}_{loc}(\R^{n-1})$; 
$v$ belongs also to the Besov space $B^2_{\infty,\infty}(\R^{n-1})$ (see \cite{Gu1}, Proposition 2.5, on account of the embedding $C(\R^{n-1}) \hookrightarrow B^0_{\infty,\infty}(\R^{n-1})$) ; moreover, for every $\gamma \in [0, 2)$ there exists $C > 0$, depending on $\phi_0$ and $\gamma$, such that
$$
\|v\|_{C^\gamma(\R^{n-1})} \leq C_1(\phi_0,\gamma) |\lambda|^{-1+ \gamma/2} \|k\|_{C(\R^{n-1})}; 
$$
finally,  if $k \in C^\beta(\R^{n-1})$, $v \in C^{2+\beta}(\R^{n-1})$. Now we fix $U_1$ open subset of $U$, with $\overline {U_1}$ contained in $U$, $x^0 \in U_1$ and $\phi \in C^{2+\beta}(\Gamma)$, with compact support in $U$, $\phi(x) = 1$ $\forall x \in U_1$. For every $h \in C(\Gamma)$, we indicate with $k$ the trivial extension of $(\phi h) \circ \Phi^{-1}$ to $\R^{n-1}$. If $\lambda$ is such that (\ref{eq2.3}) is uniquely solvable for every $k$ in $C(\R^{n-1})$, we set 
\begin{equation}
[S(x^0, \lambda) h](x) := \phi(x) v(\Phi(x)), \quad x \in \Gamma, 
\end{equation}
with $v$ solving (\ref{eq2.3}). We observe that

$(\alpha_1)$ : $S(x^0, \lambda) h \in \cap_{1\leq p < \infty} W^{2,p}(\Gamma) \cap B^2_{\infty,\infty}(\Gamma)$; 

$(\alpha_2)$: for every $\phi_0 \in [0, \pi)$, $\forall \gamma \in [0, 2)$, there exists $ C_2(\phi_0,\gamma)$ such that, if $|Arg(\lambda)| \leq \phi_0$ and $|\lambda| \geq R_1(\phi_0)$,  
$$
\|S(x^0, \lambda) h\|_{C^\gamma(\Gamma)} \leq C_2(\phi_0,\gamma) |\lambda|^{-1+\gamma/2} \|h\|_{C(\Gamma)};
$$

$(\alpha_3)$: $(\lambda - L) S(x^0, \lambda) h = h$ in $U_1$; 

$(\alpha_4)$: if (\ref{eq2.2A}) is satisfied, for some $g \in \cap_{1 \leq p < \infty}W^{2,p}(\Gamma)$, $h \in C(\Gamma)$ and $g$ vanishes outside $U_1$, then $g = S(x^0, \lambda) h$.

In  fact, the trivial extension of $g \circ \Phi^{-1}$ solves (\ref{eq2.3}), with $k$ trivial extension of $h \circ \Phi^{-1}$. 

$(\alpha_5)$: If $h \in C^\beta(\Gamma)$, $S(x^0, \lambda) h  \in C^{2+\beta}(\Gamma)$. 

Now we fix, for every $x \in \Gamma$, neighbourhoods $U(x)$, $U_1(x)$ of $x$ as before. As $\Gamma$ is compact, there exist $x_1, \dots, x_N$ in $\Gamma$ such that $\Gamma = \cup_{j=1}^N U_1(x_j)$. 

Let $\phi_0 \in [0, \pi)$, $\lambda \neq 0$, $|Arg(\lambda)|\leq \phi_0$. We show that, if $g \in \cap_{1 \leq p < \infty}W^{2,p}(\Gamma)$, it solves (\ref{eq2.2A}) with $h \equiv 0$ and $|\lambda|$ is sufficiently large, then $g \equiv 0$. In fact, let $(\phi_j)_{j=1}^N$ be a $C^{2+\beta}-$
partition of unity in $\Gamma$, with $supp(\phi_j) \subseteq U_1(x_j)$, for each $j \in \{1, \dots, N\}$. Observe that
$$
(\lambda - L) (\phi_j g) = [\phi_j; L] g, 
$$
where we have indicated with $[\phi_j; L]$ the commutator $\phi_j L - L (\phi_j \cdot)$, which is a differential operator of order one. As $(\phi_j g)(x) = 0$ outside $U_1(x_j)$, we deduce from $(\alpha_4)$, if $|\lambda|$ is sufficiently large, 
$$
\phi_j g = S(x_j, \lambda)([\phi_j; L] g). 
$$
So, from $(\alpha_2)$, 
$$
\|g\|_{C^1(\Gamma)} \leq \sum_{j=1}^N \|\phi_j g\|_{C^1(\Gamma)} \leq C_1 |\lambda|^{-1/2} \|[\phi_j; L] g\|_{C(\Gamma)} \leq C_2  |\lambda|^{-1/2} \|g\|_{C^1(\Gamma)}, 
$$
implying $g \equiv 0$ if $|\lambda|$ is sufficiently large. 

Next, we show that, if $|\lambda|$ is large enough,(\ref{eq2.2A}) is solvable for every $h \in C(\Gamma)$. This time we fix, for each $j \in \{1, \dots, N\}$, $\psi_j \in C^{2+\beta} (\Gamma)$, vanishing outside $U_1(x_j)$ and such that $\sum_{j=1}^N \psi_j(x)^2 = 1$ $\forall x \in \Gamma$. We look for $g$ in the form
$$
g = \sum_{j=1}^N \psi_j S(x_j, \lambda)(\psi_j \tilde h),
$$
for some $\tilde h \in C(\Gamma)$. Again observing that $\psi_j S(x_j, \lambda)(\psi_j \tilde h)$ vanishes outside $U_1(x_j)$ and that 
$$
(\lambda - L) [\psi_j S(x_j, \lambda)(\psi_j \tilde h)] = \psi_j^2 \tilde h + [\psi_j; L] [S(x_j, \lambda)(\psi_j \tilde h)], 
$$
we deduce 
$$
(\lambda - L) g = \tilde h + \sum_{j=1}^N [\psi_j; L] [S(x_j, \lambda)(\psi_j \tilde h)]. 
$$
So, we have to choose $\tilde h$ in such a way that 
\begin{equation}\label{eq2.5}
\tilde h + \sum_{j=1}^N [\psi_j; L] [S(x_j, \lambda)(\psi_j \tilde h)] = h. 
\end{equation}
This is uniquely possible if $|\lambda|$ is sufficiently large, because
$$
\|\sum_{j=1}^N [\psi_j; L] [S(x_j, \lambda)(\psi_j \tilde h)]\|_{C(\Gamma)} \leq C_0 \sum_{j=1}^N \|S(x_j, \lambda)(\psi_j \tilde h)]\|_{C^1(\Gamma)} \leq C_1 |\lambda|^{-1/2} \|\tilde h\|_{C(\Gamma)}.
$$
So, if $C_1 |\lambda|^{-1/2} \leq \frac{1}{2}$, we deduce from (\ref{eq2.5}) 
$$
\|\tilde h\|_{C(\Gamma)} \leq 2 \|h\|_{C(\Gamma)}, 
$$
which, together with $(\alpha_2)$, implies (I)-(II). 

Finally, if $h \in C^\beta(\Gamma)$, as $\sum_{j=1}^N [\psi_j; L] [S(x_j, \lambda)(\psi_j \tilde h) \in \cap_{\gamma < 1} C^\gamma(\Gamma)$, $\tilde h \in C^\beta(\Gamma)$, so that $g \in C^{2+\beta}(\Gamma)$ by $(\alpha_5)$. 

\end{proof}

As a simple consequence of Theorem \ref{th2.1}, we deduce the following

\begin{proposition} \label{pr2.2}
Suppose that (A1) and (A4) hold. Define the following operator $\mathcal L$: 
$$
\left\{\begin{array}{l}
D(\mathcal L) = \{g \in \cap_{1 \leq p < \infty} W^{2,p}(\Gamma) : Lg \in C(\Gamma)\}, \\ \\
\mathcal L g = Lg. 
\end{array}
\right.
$$
Then: 

(I) $D(\mathcal L) \subseteq B^2_{\infty, \infty} (\Gamma)$; 

(II) $\mathcal L$ is the infinitesimal generator of an analytic semigroup in $C(\Gamma)$; 

(III) $\forall \theta \in (0, 1) \setminus \{\frac{1}{2}\}$ the real interpolation space $(C(\Gamma), D(\mathcal L))_{\theta,\infty}$ coincides with the space $C^{2\theta}(\Gamma)$, with equivalent norms.
\end{proposition}

\begin{proof} 
(I) and (II) immediately follow from Theorem \ref{th2.1}. We observe also that  $D(\mathcal L)$ contains $C^2(\Gamma)$ and so it is dense in $C(\Gamma)$. 

Concerning (III), we have, on account of Lemma \ref{le1.4}, 
$$
C^{2\theta}(\Gamma) = (C(\Gamma), C^2(\Gamma))_{\theta,\infty} \subseteq (C(\Gamma), D(\mathcal L))_{\theta,\infty} \subseteq (B^0_{\infty,\infty}(\Gamma), B^2_{\infty,\infty}(\Gamma))_{\theta,\infty} = 
B^{2\theta}_{\infty,\infty}(\Gamma) \subseteq C^{2\theta}(\Gamma). 
$$

\end{proof}

Now we are able to study the parabolic system (\ref{eq2.2}): 

\begin{proposition}\label{pr2.3}
 Suppose that (A1) and (A4) hold. Then the following conditions are necessary and sufficient, in order that (\ref{eq2.2}) have a unique solution $g$ in $C^{1+\beta/2, 2+\beta}((0, T) \times \Gamma)$:

(a) $h \in C^{\beta/2,\beta}((0, T) \times \Gamma)$;

(b) $g_0 \in C^{2+\beta}(\Gamma)$. 

\end{proposition}

\begin{proof} If  $g \in C^{1+\beta/2, 2+\beta}((0, T) \times \Gamma)$, $g \in  C^{\beta/2}((0, T); C^2(\Gamma))$ and $D_tg \in B([0, T]; C^\beta(\Gamma))$, by Lemma \ref{le2.7B}. So (a) is necessary. (b) is obviously necessary. 

On the other hand, assume that (a) and (b) hold. By Proposition \ref{pr2.2} and Proposition 4.1.2 in \cite{Lu1},   the only possible solution is the mild solution
$$
g(t,\cdot):= e^{t \mathcal L} g_0 + \int_0^t e^{(t -s)\mathcal L} h(s,\cdot) ds,
$$
indicating with  $(e^{t \mathcal L})_{t \geq 0}$ the semigroup generated by $\mL$ in $C(\Gamma)$. By Theorem 4.3.1 in \cite{Lu1}, the following conditions are necessary and sufficient, in order that $g \in C^{1+\beta/2}((0, T); C(\Gamma))$: 
$g_0 \in D(\mathcal L)$, $h \in C^{\beta/2}((0, T); C(\Gamma))$, $\mathcal L g_0 + h(0,\cdot) \in (C(\Gamma), D(\mathcal L))_{\beta/2, \infty}$. The two first conditions clearly follow from (a)-(b). The third follows from the identity 
$(C(\Gamma), D(\mathcal L))_{\beta/2, \infty} = C^\beta(\Gamma)$, by Proposition \ref{pr2.2} (III). Finally, , by Corollary 4.3.9 in \cite{Lu1}, $g$ is bounded with values in $D(\mathcal L)$ and $\mathcal L g$ is bounded with values in $(C(\Gamma), D(\mathcal L))_{\beta/2, \infty} = C^\beta(\Gamma)$ if and only if $g_0 \in D(\mathcal L)$,  $\mathcal L g_0 \in (C(\Gamma), D(\mathcal L))_{\beta/2, \infty} = C^\beta(\Gamma)$, $h \in C((0, T); C(\Gamma)) \cap B([0, T]; (C(\Gamma), D(\mathcal L))_{\beta/2, \infty}$ and this follows again from Proposition \ref{pr2.2} (III). As $\mathcal L g$ is bounded with values in $C^\beta(\Gamma)$, we conclude, by Theorem \ref{th2.1} (III), that $g$ is bounded with values in 
$C^{2+\beta}(\Gamma)$. 
\end{proof} 

\section{The problem}\label{se3}

\setcounter{equation}{0}

We introduce for future reference the following conditions $(A2')$, which is nothing but $(A2)$ in the particular case that $A(t,x,D_x)$ does not depend on $t$: 

\medskip
{(A2') 
$A(x,D_x) = \sum_{|\alpha| \leq 2} a_\alpha(x) D_x^\alpha$, with $a_\alpha \in C^{\beta} (\Omega)$; in case $|\alpha| = 2$, $a_\alpha$ is real valued and there exists $\nu > 0$ such that, $\forall  x \in  \overline \Omega$, 
$\forall \xi \in \R^n$, $\sum_{|\alpha| = 2} a_\alpha(x) \xi^\alpha \geq \nu |\xi|^2$. 

}
\medskip

We recall the following classical result (see \cite{Lu1}, Theorem 5.1.15):

\begin{theorem}\label{th3.1}
Suppose that (A1) and (A2') hold. Then the following conditions are necessary and sufficient, in order that the system
\begin{equation}\label{eq3.1A}
\left\{\begin{array}{ll}
D_tu(t,x) = A(x, D_x) u(t,x) + f(t,x), & (t,x) \in [0, T] \times \Omega, \\ \\
u(t,x') = g(t,x'), & (t,x') \in [0, T] \times \Gamma, \\ \\
u(0,x) = u_0(x), & x \in \Omega. 
\end{array}
\right. 
\end{equation}
have a unique solution $u$ in $C^{1+\beta/2, 2+\beta}((0, T) \times \Omega)$:

(a) $f \in C^{\beta/2,\beta}((0, T) \times \Omega)$; 

(b) $g \in C^{1+\beta/2, 2+\beta}((0, T) \times \Gamma)$; 

(c) $u_0 \in C^{2+\beta}(\Omega)$; 

(d) $u_{0|\Gamma} = g(0,\cdot)$, $A(x', D_x) u_0(x') + f(0,x') = D_tg(0,x')$ $\forall x' \in \Gamma$. 
\end{theorem}

As a consequence, we have the following simple

\begin{theorem}\label{th3.2}
Suppose that (A1), (A2'), (A4') hold. Consider the system
\begin{equation}\label{eq3.1}
\left\{\begin{array}{ll}
D_tu(t,x) = A(x, D_x) u(t,x) + f(t,x), & (t,x) \in [0, T] \times \Omega, \\ \\
D_tu(t,x')  - L u(t,x') = h(t,x'), & (t,x') \in [0, T] \times \Gamma, \\ \\
u(0,x) = u_0(x), & x \in \Omega. 
\end{array}
\right. 
\end{equation}
Then the following conditions are necessary and sufficient, in order that (\ref{eq3.1}) have a unique solution $u$ in $C^{1+\beta/2, 2+\beta}((0, T) \times \Omega)$:

(a) $f \in C^{\beta/2,\beta}((0, T) \times \Omega)$; 

(b) $h \in C^{\beta/2,\beta}((0, T) \times \Gamma)$; 

(c) $u_0 \in C^{2+\beta}(\Omega)$; 

(d) $A(x', D_x) u_0(x') + f(0,x') = L(u_{0|\Gamma})(x') + h(0,x')$ $\forall x' \in \Gamma$. 
\end{theorem} 

\begin{proof}

We begin by showing that (a)-(d) are necessary. In fact, (a) and (c) are necessary by Theorem \ref{th3.1}. Moreover, $g:= u_{(0, T) \times \Gamma)}$ should belong to $C^{1+\beta/2, 2+\beta}((0, T) \times \Gamma)$ and satisfy the equation
$$
D_t g(t,x')  - L g(t,x') = h(t,x'), \quad (t,x') \in [0, T] \times \Gamma, 
$$
so that, by Proposition \ref{pr2.3}, (b) is necessary. Finally, $\forall x' \in \Gamma$, 
$$A(x', D_x) u_0(x') + f(0,x') = D_t u(0,x') = L(u_{0|\Gamma})(x') + h(0,x'). $$

Next, we show that (a)-(d) are sufficient. If a solution $u$ with the declared regularity exists, $g:= u_{|(0, T) \times \Gamma}$ belongs to $C^{1+\beta/2, 2+\beta}((0, T) \times \Gamma)$ and satisfies the system (\ref{eq2.2}), with $g_0 = u_{0|\Gamma} \in C^{2+\beta}(\Gamma)$. By Proposition \ref{pr2.3}, there is a unique solution $g$ in $C^{1+\beta/2, 2+\beta}((0, T) \times \Gamma)$. So $u$ must be the solution to (\ref{eq3.1A}).
 The conditions (a)-(d) in Theorem \ref{th3.1} are all satisfied. In fact, for example, $\forall x' \in \Gamma$, 
$$A(x', D_x) u_0(x') + f(0,x') = L(u_{0|\Gamma})(x') + h(0,x') = D_tg(0,x').$$
We conclude that there exists a unique solution $u$ in $C^{1+\beta/2, 2+\beta}((0, T) \times \Omega)$ and this completes the proof. 
\end{proof}

Now we need some estimates of the solution $u$ to (\ref{eq3.1})  in different functional settings. They are very similar to those contained in \cite{Gu2}, Lemma 2.2. We shall indicate $A(x,D_x)$ with $A$. 

\begin{lemma}\label{le4.1}
Assume that (A1), (A2'), (A4') hold, $T_0 \in \R^+$ and $T \leq T_0$. Suppose that $f$, $h$, $u_0$ satisfy conditions (a)-(d) in the statement of Theorem \ref{th3.2}. Let $u$ be the solution in $C^{1+\beta/2,2+\beta}((0, T) \times \Omega)$ 
 of (\ref{eq3.1}). Then: 
 
 (I) there exists $C(A,L,T_0)$ in $\R^+$, such that
$$
\begin{array}{c}
\|u\|_{C^{1+\beta/2,2+\beta}((0, T) \times \Omega))} + \|u\|_{C^{\beta/2}((0, T); C^2(\Omega))} + \|u\|_{C^{\frac{1+\beta}{2}}((0, T); C^1(\Omega))}Ê\\ \\
\leq C(A,L,T_0) (\|f\|_{C^{\beta/2,\beta}((0, T) \times \Omega)} + \|u_0\|_{C^{2+\beta}(\Omega)} + \|h\|_{C^{\beta/2,\beta}((0, T) \times \Gamma)}). 
\end{array}
$$
(II) Suppose that $u_0 = 0$. Then, if $0 \leq \theta \leq 1$, there exists $C(A,L,T_0,\theta)$ in $\R^+$ such that
\begin{equation}\label{eq4.2}
\|u\|_{C^\theta((0, T); C(\Omega))} \leq C(A,L,T_0,\theta) T^{1-\theta} (\|f\|_{C^{\beta/2,\beta}((0, T) \times \Omega)} + \|h\|_{C^{\beta/2,\beta}((0, T) \times \Gamma)}),
\end{equation}
if $0 \leq \xi \leq 2+\beta$, there exists $C(A,L,T_0,\xi)$ in $\R^+$ such that
\begin{equation}\label{eq4.3}
\|u\|_{B((0, T); C^\xi(\Omega))} \leq C(A,L,T_0,\xi) T^{\frac{2+\beta-\xi}{2+\xi}} (\|f\|_{C^{\beta/2,\beta}((0, T) \times \Omega)} + \|h\|_{C^{\beta/2,\beta}((0, T) \times \Gamma)}). 
\end{equation}
(III) Again supposing $u_0 = 0$, there exists $C(A,L, T_0)$ in $\R^+$ such that  
$$
\|u\|_{C^{\beta/2}((0, T); C^1(\Omega))} \leq C(A,L,T_0) T^{1/2} (\|f\|_{C^{\beta/2,\beta}((0, T) \times \Omega)} + \|h\|_{C^{\beta/2,\beta}((0, T) \times \partial \Omega)}).
$$
\end{lemma}

\begin{proof} We extend $f$ and $h$ to elements $\tilde f$ and $\tilde h$ in (respectively) $C^{\beta/2, \beta}((0, T_0) \times \Omega)$ and $C^{\beta/2,\beta}$ $((0, T_0) \times \partial \Omega)$: we set
$$
\tilde f(t,x) = \left\{\begin{array}{ll}
f(t,x) & \mbox{ if } 0 \leq t \leq T, x \in \Omega, \\ \\
f(T,x) & \mbox{ if } 0 \leq  T \leq t \leq T_0, x \in \Omega,
\end{array}
\right.
$$
$$
\tilde h(t,x') = \left\{\begin{array}{ll}
h(t,x') & \mbox{ if } 0 \leq t \leq T, x' \in \Gamma, \\ \\
h(T,x') & \mbox{ if } 0 \leq  T \leq t \leq T_0, x' \in \Gamma. 
\end{array}
\right.
$$
We denote with $\tilde u$ the solution to 
$$
\left\{\begin{array}{l}
D_t \tilde u(t,\xi) = A(x,D_x) \tilde u(t,x) + \tilde f(t,x), \quad t \in [0, T_0], x \in \Omega, \\ \\
\tilde u(0, x) = u_0(x), \quad x \in \Omega, \\ \\
D_t \tilde u(t,x') - L \tilde u(t,x') =  \tilde h(t,x'), \quad t \in [0, T_0], x'  \in \Gamma. 
\end{array}
\right.
$$
Clearly, $\tilde u$ is an extension of $u$. So
$$
\begin{array}{c}
\|u\|_{C^{1+\beta/2, 2+\beta}((0, T) \times \Omega)} + \|u\|_{C^{\beta/2}((0, T); C^2 (\Omega))} + \|u\|_{C^{\frac{1+\beta}{2}}((0, T); C^1 (\Omega))}\\ \\
\leq \|\tilde u\|_{C^{1+\beta/2, 2+\beta}((0, T_0) \times \Omega)} + \|\tilde u\|_{C^{\beta/2}((0, T_0); C^2 (\Omega))} + \|\tilde u\|_{C^{\frac{1+\beta}{2}}((0, T_0); C^1 (\Omega))}
  \\ \\
\leq C(A,L,T_0)(\|\tilde f\|_{C^{\beta/2, \beta}((0, T_0) \times \Omega)} + \|u_0\|_{C^{2+\beta}(\Omega)} + \|\tilde h\|_{C^{\beta/2,\beta}((0, T_0) \times \Gamma)}) \\ \\
= C(A,L,T_0)(\|f\|_{C^{\beta/2, \beta}((0, T) \times \Omega)} + \|u_0\|_{C^{2+\beta}(\Omega)} + \|h\|_{C^{\beta/2,\beta}((0, T) \times \Gamma)}). 
\end{array}
$$
So (I) is proved.

Concerning (II), we begin by considering the case $\theta = 0$. If $0 \leq t \leq T$, we have
$$
u(t,\cdot) = \int_0^t D_s u(s,\cdot) ds, 
$$
so that
$$
\begin{array}{c}
\|u(t,\cdot)\|_{C(\Omega)} \leq \int_0^t \|D_s u(s,\cdot)\|_{C(\Omega)} ds \leq t \|u\|_{C^{1+\beta/2}((0, T); C(\Omega))} \\ \\
 \leq C(A,L,T_0) T (\|f\|_{C^{\beta/2, \beta}((0, T) \times \Omega)} + \|h\|_{C^{\beta/2,\beta}((0, T) \times \Gamma)}),
\end{array}
$$
employing (I). The cases $0 < \theta \leq 1$ follow from the foregoing with $\theta = 0$ and the fact that $C^\theta((0, T); C(\Omega)) \in J_\theta(C((0, T); C(\Omega)),C^1((0, T); C(\Omega))$. 
So (\ref{eq4.2}) is proved. (\ref{eq4.3}) follows from (\ref{eq4.2}) with $\theta = 0$, (I) and the fact that $C^\theta(\Omega) \in J_{\theta/(2+\beta)}(C(\Omega), C^{2+\beta}(\Omega))$. 

It remains to consider (III). We have
$$
\begin{array}{c}
\|u\|_{C^{\beta/2}((0, T); C^1(\Omega))} = \max\{\|u\|_{C((0, T); C^1(\Omega))}, [u]_{C^{\beta/2}((0, T); C^1(\Omega))}\}  \\ \\
\leq \|u\|_{C((0, T); C^1(\Omega))}  + T^{1/2} [u]_{C^{(1+\beta)/2}((0, T); C^1(\Omega))},
\end{array}
$$
and the conclusion follows from (I) and (II). 

\end{proof}

\begin{remark}\label{re3.4}
{\rm A revision of the proof of (II)-(III) in Lemma \ref{le4.1} shows that the dependence of the constants $C(A,L,T_0,\theta)$ in (\ref{eq4.2}),  $C(A,L,T_0,\xi)$ in (\ref{eq4.3}) and $C(A,L,T_0)$ in (III) on $A$ and $L$ is through the constant 
$C(A,L,T_0)$ in (I). In other words, if we consider all   operators $A$ and $L$ such that (I) holds with the same constant $C(T_0)$, the constants in (\ref{eq4.2}), (\ref{eq4.3}) and (IV) can be chosen independently of $A$ and $L$. 

}
\end{remark}

\begin{lemma}\label{le3.4}
We consider a system in the form
\begin{equation}\label{eq3.5}
\left\{\begin{array}{ll}
D_tu(t,x) = A(x, D_x) u(t,x) + \A (t) u(t,\cdot) + f(t,x), & (t,x) \in [0, T] \times \Omega, \\ \\
D_t u(t,x') + \mB (t)u(t,x')  - L u(t,x') = h(t,x'), & (t,x') \in [0, T] \times \partial \Omega, \\ \\
u(0,x) = u_0(x), & x \in \Omega. 
\end{array}
\right. 
\end{equation}
with the following assumptions:

(a) (A1), (A2') and (A4') are satisfied; 

(b) $\forall t \in [0, T]$ $\A(t) \in \mL(C^2(\Omega); C(\Omega)) \cap \mL(C^{2+\beta}(\Omega); C^\beta(\Omega))$ and, for certain $\delta$, $M$ in $\R^+$, $\forall u \in C^{2+\beta}(\Omega)$, 
\begin{equation}\label{eq3.6B}
\|\A(t)u\|_{C(\Omega)} \leq \delta \|u\|_{C^2(\Omega)} + M \|u\|_{C^{1}(\Omega)}, 
\end{equation}
\begin{equation}\label{eq3.7B}
\|\A\|_{C^{\beta/2}([0, T]; \mL(C^2(\Omega); C(\Omega))} \leq M, 
\end{equation}
\begin{equation}\label{eq3.8B}
\|\A(t)u\|_{C^\beta(\Omega)} \leq \delta \|u\|_{C^{2+\beta}(\Omega)} + M \|u\|_{C^{2}(\Omega)};
\end{equation}

(c) $\forall t \in [0, T]$ $\mB(t) \in \mL(C^2(\Omega); C(\Gamma)) \cap \mL(C^{2+\beta}(\Omega); C^\beta(\Gamma))$ and, for certain $\delta$, $M$ in $\R^+$, $\forall u \in C^{2+\beta}(\Omega)$, 
\begin{equation}\label{eq3.9B}
\|\mB(t)u\|_{C(\Gamma)} \leq \delta \|u\|_{C^2(\Omega)} + M \|u\|_{C^{1}(\Omega)}, 
\end{equation}
\begin{equation}\label{eq3.10B}
\|\mB\|_{C^{\beta/2}([0, T]; \mL(C^2(\Omega); C(\Gamma))} \leq M, 
\end{equation}
\begin{equation}\label{eq3.11B}
\|\mB(t)u\|_{C^\beta(\Gamma)} \leq \delta \|u\|_{C^{2+\beta}(\Omega)} + M \|u\|_{C^{2}(\Omega)}. 
\end{equation}

Then: 

\medskip

(I) there exists $\delta_0$ in $\R^+$, depending only on $A $ and $L$,  such that, if $\delta \leq \delta_0$, $f \in C^{\beta/2, \beta}((0, T) \times \Omega)$,  $h \in C^{\beta/2, \beta}((0, T) \times \Gamma)$, $u_0 \in C^{2+\beta}(\Omega)$, and 
\begin{equation}\label{eq3.6A}
A(x', D_x) u_0(x')+ \A (0) u_0(x')  + f(0,x') = - \mB (0)u_0(x')  + L (u_{0|\Gamma})(x') + h(0,x'), \quad \forall x' \in \Gamma, 
\end{equation}
(\ref{eq3.5}) has a unique solution $u$ in $C^{1+\beta/2,2+\beta}((0, T) \times \Omega)$. 

(II) Suppose that $T_0 \in \R^+$, (b) and (c) are satisfied, replacing $T$ with $T_0$,  $\delta \leq \delta_0$, with $\delta_0$ as in (I). Take $T \in (0, T_0]$. Then: 

(II) there exists $C(T_0, A,L,\delta, M)$ in $\R^+$, such that
$$
\begin{array}{c}
\|u\|_{C^{1+\beta/2,2+\beta}((0, T) \times \Omega)} + \|u\|_{C^{\beta/2}((0, T); C^2 (\Omega))} + \|u\|_{C^{\frac{1+\beta}{2}}((0, T); C^1 (\Omega))} \\ \\
 \leq C(T_0, A,L,\delta, M) (\|f\|_{C^{\beta/2, \beta}((0, T) \times \Omega)} + \|u_0\|_{C^{2+\beta}(\Omega)} + \|h\|_{C^{\beta/2, \beta}((0, T) \times \Gamma)}). 
\end{array}
$$

(III) Suppose that $u_0 = 0$. Then, if $0 \leq \theta \leq 1$, 
$$
\|u\|_{C^\theta((0, T); C(\Omega))} \leq C(T_0, A,L,\delta, M,\theta) T^{1-\theta} (\|f\|_{C^{\beta/2, \beta}((0, T) \times \Omega)}  + \|h\|_{C^{\beta/2, \beta}((0, T) \times \Gamma)}), 
$$
if $0 \leq \theta \leq 2+\beta$, 
$$
\|u\|_{B((0, T); C^\beta(\Omega))} \leq C(T_0, A,L,\delta, M,\beta) T^{\frac{2+\beta - \theta}{2+\beta}} (\|f\|_{C^{\beta/2, \beta}((0, T) \times \Omega)}  + \|h\|_{C^{\beta/2, \beta}((0, T) \times \Gamma)});
$$
$$
\|u\|_{C^{\beta/2}((0, T); C^1(\Omega))} \leq C(T_0, A,L,\delta, M) T^{1/2} (\|f\|_{C^{\beta/2, \beta}((0, T) \times \Omega)}  + \|h\|_{C^{\beta/2, \beta}((0, T) \times \Gamma)});
$$
\end{lemma}

\begin{proof} We prove the result in several steps.

{\it $(\alpha)$ We consider the case $u_0 = 0$ and prove an a priori estimate of $u$, if $T$ is sufficiently small. }

In this case, (\ref{eq3.6A}) becomes 
$$
f(0,x') =  h(0,x'), \quad \forall x' \in \Gamma.
$$
So suppose that $u \in C^{1+\beta/2,2+\beta}((0, T) \times \Omega)$ solves (\ref{eq3.5}), in case $u_0 = 0$. We set
$$
\A u(t):= \A(t)u(t), \quad \mB u(t):= \mB(t)u(t). 
$$
It is easily seen that $\A u \in C^{\beta/2,\beta}((0, T) \times \Omega)$ and $\mB u \in C^{\beta/2,\beta}((0, T) \times \Gamma)$.
Moreover, if $0 \leq s < t \leq T$, 
$$
\begin{array}{c}
\|\A (t) u(t) - \A (s) u(s)\|_{C(\Omega)} \leq \|\A (t) - \A(s)) u(t)\|_{C(\Omega)} + \|\A (s) (u(t) - u(s))\|_{C(\Omega)} \\ \\
\leq (t-s)^{\beta/2} [M \|u\|_{C((0, T); C^2(\Omega))} + \delta  \|u\|_{C^{\beta/2}((0, T); C^2(\Omega))} + M \|u\|_{C^{\beta/2}((0, T); C^1(\Omega))}],
\end{array}
$$
$$
\|\A u\|_{B((0, T); C^\beta(\Omega))} \leq \delta \|u\|_{B((0, T); C^{2+\beta}(\Omega))} + M  \|u\|_{B((0, T); C^{2}(\Omega))}. 
$$
Analogously, one can show that 
$$
\begin{array}{c}
\|\mB (t) u(t) - \mB (s) u(s)\|_{C(\Gamma)} \\ \\
 \leq 
\ (t-s)^{\beta/2} [M \|u\|_{C((0, T); C^2(\Omega))} + \delta  \|u\|_{C^{\beta/2}((0, T); C^2(\Omega))} + M \|u\|_{C^{\beta/2}((0, T); C^1(\Omega))}],
\end{array}
$$
$$
\|\mB u\|_{B((0, T); C^\beta(\Gamma))} \leq \delta \|u\|_{B((0, T); C^{2+\beta}(\Omega))} + M  \|u\|_{B((0, T); C^{2}(\Omega))}. 
$$
So, from Lemma \ref{le4.1}, if (say) $T \leq 1$, we deduce 
$$
\begin{array}{c}
\|u\|_{C^{1+\beta/2,2+\beta}((0, T) \times \Omega)} + \|u\|_{C^{\beta/2}((0, T); C^2(\Omega))}  +  T^{-\frac{\beta}{2+\beta}} \|u\|_{C((0, T); C^2(\Omega))} +  T^{-\frac{1}{2}} \|u\|_{C^{\beta/2}((0, T); C^1(\Omega))} \\ \\
  \leq C(A,L) [ \|f\|_{C^{\beta/2,\beta}((0, T) \times \Omega)} + \|h\|_{C^{\beta/2,\beta}((0, T) \times \Gamma)} + M (\|u\|_{C((0, T); C^2(\Omega))} + \|u\|_{C^{\beta/2}((0, T); C^1(\Omega))}) \\ \\
  + \delta (\|u\|_{C^{1+\beta/2,2+\beta}((0, T) \times \Omega)} + \|u\|_{C^{\beta/2}((0, T); C^2(\Omega))})]. 
\end{array}
$$
So, if we assume that $\delta \leq \delta_0$ and $T$ is so small that
\begin{equation}\label{eq3.7A}
C(A,B) \delta_0 \leq \frac{1}{2}, C(A,L) M \leq \min\{\frac{1}{2T^{\frac{\beta}{2+\beta}}}, \frac{1}{2T^{1/2}}\}, 
\end{equation}
we obtain the estimate
\begin{equation}\label{eq3.7} 
\|u\|_{C^{1+\beta/2,2+\beta}((0, T) \times \Omega)} \leq 2C(A,L) (\|f\|_{C^{\beta/2,\beta}((0, T) \times \Omega)} + \|h\|_{C^{\beta/2,\beta}((0, T) \times \Gamma)}). 
\end{equation}

\medskip

{\it $(\beta)$ We show that, in case $u_0 = 0$, if $T$ is sufficiently small in such a way that (\ref{eq3.7}) holds, again in case $u_0 = 0$, (\ref{eq3.5}) 
has a unique solution $u \in C^{1+\beta/2,2+\beta}((0, T) \times \Omega)$. 
}

This is a consequence of the continuation principle (see Proposition \ref{pr1.20B}): define 
$$
X_\tau:= \{u \in C^{1+\beta/2,2+\beta}((0, \tau) \times \Omega) : u(0,\cdot) = 0\}
$$
which is a closed subspace of $C^{1+\beta/2,2+\beta}((0, \tau) \times \Omega)$, 
$$
Y_\tau  := \{(f,h) \in C^{\beta/2,\beta}((0, \tau) \times \Omega) \times C^{\beta/2,\beta}((0, \tau) \times \Gamma) : f(0,\cdot)_{|\Gamma} = h(0,\cdot)\}, 
$$
which is a closed subspace of $C^{\beta/2,\beta}((0, \tau) \times \Omega) \times C^{\beta/2,\beta}((0, \tau) \times \Gamma)$. If $\rho \in [0, 1]$, consider the operator $T_\rho : X_\tau \to Y_\tau$, 
$$
T_\rho u := (D_t u - Au - \rho \A u, D_t u_{|(0, T) \times \Gamma} - L (u_{|(0, T) \times \Gamma}) - \rho \mB u). 
$$
$T_0$ is a linear and topological isomorphism between $X_\tau$ and $Y_\tau$, by Theorem \ref{th3.2}. So $(\beta)$ follows from the a priori estimates $(\alpha)$. 

\medskip

{\it $(\gamma)$ We show that, if $T$ is sufficiently small in such a way that (\ref{eq3.7}) holds, if $f \in C^{\beta/2,\beta}((0, T) \times \Omega)$, $h \in C^{\beta/2,\beta}((0, T) \times \Gamma)$, $u_0 \in C^{2+\beta}(\Omega)$ and (\ref{eq3.6A}) holds, \ref{eq3.5}) 
has a unique solution $u \in C^{1+\beta/2,2+\beta}((0, T) \times \Omega)$. }

We take $v(t,x) := u(t,x) - u_0$ as new unknown. $v$ should solve the system 
$$
\left\{\begin{array}{ll}
D_t v(t,x) = A(x, D_x) v(t,x) + \A (t) v(t,\cdot) + f(t,x) + A(x, D_x) u_0(x) + \A (t)u_0, & (t,x) \in [0, T] \times \Omega, \\ \\
D_t v(t,x') + \mB (t)v(t,x')  - L v(t,x') = h(t,x') - \mB (t)u_0 + L(u_{0|\Gamma}) , & (t,x') \in [0, T] \times \partial \Omega, \\ \\
v(0,x) = 0, & x \in \Omega. 
\end{array}
\right. 
$$
to which the conclusion in $(\beta)$ is applicable. 

\medskip

{\it $(\delta)$ Proof of   (I). } 

We begin by showing the uniqueness of a solution. So suppose that $u \in C^{1+\beta/2,2+\beta}((0, T) \times \Omega)$ solves (\ref{eq3.5}) in case $f \equiv  0$, $h \equiv 0$ and $u_0 = 0$. Of course we want to show that $u \equiv 0$. 
Suppose the contrary and set
$$
\tau:= \inf \{t \in [0, T]: u(t,\cdot) \neq 0\}. 
$$
Then $0 \leq \tau < T$ and, by continuity, $u(\tau, \cdot) = 0$. Define $u_1(t, \cdot):= u(\tau + t,\cdot)$. Then $u_1$ is a solution to the system 
$$
\left\{\begin{array}{ll}
D_tu_1(t,x) = A(x, D_x) u_1(t,x) + \A (\tau + t) u_1(t,\cdot), & (t,x) \in [0, T-\tau] \times \Omega, \\ \\
D_t u_1(t,x') + \mB (\tau + t)u_1(t,x')  - L u_1(t,x') = 0, & (t,x') \in [0, T-\tau] \times \partial \Omega, \\ \\
u_1(0,x) = 0, & x \in \Omega. 
\end{array}
\right. 
$$
Clearly, the families of operators $\{\A (\tau+t) : t \in [0, T - \tau]\}$ and $\{\mB (\tau+t) : t \in [0, T - \tau]\}$ satisfy the conditions (b) and (c) with the same constants $\delta$ and $M$. We deduce from $(\beta)$ that, for some $T_1 \in (0, T - \tau]$, $u_{1|(0, T_1) \times \Omega} = 0$, so that $u_{|(0, \tau + T_1) \times \Omega} = 0$, which is in contradiction with the definition of $\tau$. 

We show the existence of a solution. We suppose that $T$ does not satisfy one of the majorities in (\ref{eq3.7A}) and replace it with $\tau \in (0, T)$, satisfying them. So we can apply $(\gamma)$, and get a solution $u_1$ in $[0, \tau] \times \Omega$. We observe that, $\forall x' \in \Gamma$, 
$$
A(x', D_x) u_1(\tau, x')+ \A (\tau) u_1(\tau,x')  + f(\tau,x') = D_tu(\tau,x') = - \mB (\tau)u_1(\tau,x')  + L u_1(\tau,x') + h(\tau,x'), 
$$
and consider the system 
\begin{equation}\label{eq3.9}
\left\{\begin{array}{ll}
D_tu_2(t,x) = A(x, D_x) u_2(t,x) + \A (\tau+t) u_2(t,x) + f(\tau+t,x), & (t,x) \in [0, \tau \wedge (T-\tau)] \times \Omega, \\ \\
D_t u_2(\tau,x') + \mB (\tau+t)u_2(t,x')  - L u_2(t,x') = h(\tau+t,x'), & (t,x') \in [0, \tau \wedge (T-\tau)] \times \Gamma, \\ \\
u_2(0,x) = u_1(\tau,x), & x \in \Omega. 
\end{array}
\right. 
\end{equation}
By $(\gamma)$, (\ref{eq3.9}) has a unique solution $u_2$ in $[0, \tau \wedge (T-\tau)] \times \Omega$. If we set 
$$
u(t,x) = \left\{\begin{array}{lll}
u_1(t,x) & {\rm if } & (t,x) \in [0, \tau] \times \overline \Omega, \\ \\
u_2(t - \tau,x) & {\rm if } & (t,x) \in [\tau, (2\tau) \wedge T] \times \overline \Omega, 
\end{array}
\right.
$$
we can easily check that $u \in C^{1+\beta/2,2+\beta}((0, ((2\tau) \wedge T) \times \Omega)$ and is a solution to (\ref{eq3.5}). In case $2\tau < T$, we can extend $u$ to a solution in $[0, (3\tau) \wedge T] \times \Omega$ and in a finite number of steps we construct a solution in $[0, T] \times \Omega$. 

\medskip

{\it $(\epsilon)$ Proof of (II)-(III)} 

It can be obtained with the same arguments in the proof of Lemma \ref{le4.1}.

\end{proof} 

\begin{remark}\label{re3.6}
{\rm A revision of the proof of Lemma \ref{le3.4}  shows that the dependence of the positive number $\delta_0$ and the constants $C(T_0,A,L,\delta,M,\theta)$ and $C(T_0, A,L,\delta, M,\beta)$ appearing in (I)-(III) on $A$ and $L$ is through the constants appearing in (I)-(III) of the statement of Lemma \ref{le4.1}. So, by Remark \ref{re3.4}, it is through the only constant $C(A,L,T_0)$ appearing in (I) of Lemma \ref{le4.1}. 
}
\end{remark}

\begin{remark}\label{re3.5}
{\rm Lemma \ref{le3.4} is applicable in case
\begin{equation}
\A(t) v(x) = \sum_{|\alpha| \leq 2} r_\alpha (t,x) D_x^\alpha v(x), \quad (t,x) \in [0, T] \times \overline \Omega, 
\end{equation}
\begin{equation}
\mB(t) v(x) =  \sum_{|\alpha| \leq 1} b_\alpha(t,x') D_x^\alpha v(x') + \mL(t) v(x'), 
\end{equation} 
Here $\mL(t)$ is a differential operator in $\Gamma$ of order not exceeding two, which, for some finite subatlas $(U_j, \Phi_j)_{j=1}^N$ of $\Gamma$, can be locally represented in the form 
\begin{equation}
\mL(t) v(x') = \sum_{|\alpha| \leq 2} l_{j\alpha}(t,x') D_y^\alpha (v \circ \Phi_j^{-1}) (\Phi_j(x')), \quad t \in [0, T], x' \in U_j. 
\end{equation}
We suppose that  $r_\alpha \in C^{\beta/2,\beta}((0, T) \times \Omega)$ $(|\alpha| \leq 2)$, $b_\alpha \in C^{\beta/2,\beta}((0, T) \times \Gamma)$ $(|\alpha| \leq 1)$, $l_{j\alpha} \in  C^{\beta/2,\beta}((0, T) \times U_j)$ ($1 \leq j \leq N$, $|\alpha|\leq 2$), 
\begin{equation}
\sum_{|\alpha| = 2} \|r_\alpha\|_{C((0, T) \times \Omega)} + \sum_{j=1}^N \sum_{|\alpha| = 2} \|l_{j\alpha}\|_{C((0, T) \times U_j)} \leq d, 
\end{equation}
\begin{equation}
\sum_{|\alpha| \leq 2} \|r_\alpha\|_{C^{\beta/2,\beta} ((0, T) \times \Omega)} + \sum_{|\alpha| \leq 1} \|b_\alpha\|_{C^{\beta/2,\beta} ((0, T) \times \Gamma)} + \sum_{j=1}^N \sum_{|\alpha| \leq 2} \|l_{j\alpha}\|_{C^{\beta/2,\beta}((0, T) \times U_j)} \leq P.
\end{equation}
Then there is a positive constant $C$, independent of the coefficients $r_\alpha$, $b_\alpha$, $l_{j\alpha}$ and of $v$, such that
\begin{equation}\label{eq3.15}
\|\A(t) v\|_{C(\Omega)} \leq C(d \|v\|_{C^2(\Omega)} + P \|v\|_{C^1(\Omega)}), 
\end{equation}
\begin{equation}
\|\A\|_{C^{\beta/2}((0, T); \mL (C^2(\Omega), C(\Omega))} \leq CP, 
\end{equation}
\begin{equation}
\|\A(t) v\|_{C^\beta(\Omega)} \leq C(d \|v\|_{C^{2+\beta}(\Omega)} + P \|v\|_{C^{2}(\Omega)}), 
\end{equation}
\begin{equation}
\|\mB(t) v\|_{C(\Gamma)} \leq C(d \|v\|_{C^2(\Omega)} + P \|v\|_{C^1(\Omega)}), 
\end{equation}
\begin{equation}
\|\mB\|_{C^{\beta/2}((0, T); \mL (C^2(\Omega), C(\Gamma))} \leq CP, 
\end{equation}
\begin{equation}
\|\mB(t) v\|_{C^\beta(\Gamma)} \leq C(d \|v\|_{C^{2+\beta}(\Omega)} + P \|v\|_{C^{2}(\Omega)}). 
\end{equation}
So we can take $\delta = C d$ and $M = CP$. 

Observe that the possibility of applying Lemma \ref{le3.4} depends only on $\delta$ and so only on $d$. $P$ has a role only in estimates envolving the solution $u$. 
}
\end{remark}

\begin{lemma}\label{le3.8}
Suppose that (A1), (A2) and (A4) hold. For each $t_0 \in [0, T]$, we consider the constants $\delta(t_0)$, $C(T_0,A(t_0),L(t_0),\delta,M)$, $C(T_0,A(t_0),L(t_0),\delta,M,\theta)$, $C(T_0,A(t_0),L(t_0),\delta,M,\beta)$ in (I)-(III) of Lemma \ref{le3.4}, taking $A = A(t_0)$ and $L = L(t_0)$. Then each of them can be chosen independently of $t_0$ in $[0, T]$. 
\end{lemma}

\begin{proof} We take $t_0, t_1$ in $[0, T]$. We set $A(x,D_x):= A(t_0,x,D_x)$, $L:= L(t_0)$, $\A(t):= A(t_1,x,D_x) - A(t_0,x,D_x)$, $\mB(t):= L(t_1) - L(t_0)$. We observe that $\A(t)$ and $\mB(t)$ are, in fact, independent of $t$. Then, $\forall
u \in C^{2+\beta}(\Omega)$, 
$$
\|\A(t)u\|_{C(\Omega)} \leq \omega(|t_1 - t_0|) \|u\|_{C^2(\Omega))}, 
$$
$$
\|\A\|_{C^{\beta/2}([0, T]; \mL(C^2(\Omega); C(\Omega))} \leq \omega(|t_1 - t_0|), 
$$
$$
\|\A(t)u\|_{C^\beta(\Omega)} \leq \omega(|t_1 - t_0|) \|u\|_{C^{2+\beta}(\Omega)} + M \|u\|_{C^{2}(\Omega)}, 
$$
$$
\|\mB(t)u\|_{C(\Gamma)} \leq \omega(|t_1 - t_0|) \|u\|_{C^2(\Omega))}, 
$$
$$
\|\mB\|_{C^{\beta/2}([0, T]; \mL(C^2(\Omega); C(\Gamma))} \leq \omega(|t_1 - t_0|), 
$$
$$
\|\mB(t)u\|_{C^\beta(\Gamma)} \leq \omega(|t_1 - t_0|) \|u\|_{C^{2+\beta}(\Omega)} + M \|u\|_{C^{2}(\Omega)}, 
$$
with $M \in \R^+$ independent of $t_0$ and $t_1$ and ${\displaystyle \lim_{r\to 0}} \omega(r) = 0$. We deduce from Lemma \ref{le3.4}  that there exists $r(t_0) > 0$, such that, if $|t_1 - t_0| < r(t_0)$,  the constants can be chosen independently of $t_1$. So the conclusion follows from the compactness of $[0, T]$. 

\end{proof}

Now we are able to prove Theorems \ref{th1.1} and \ref{th1.2}. 

\medskip

{\bf Proof of Theorem \ref{th1.1}. } Let $t_0 \in [0, T)$. We consider the system

\begin{equation}\label{eq3.21}
\left\{\begin{array}{ll}
D_tv(t,x) = A(t_0+t,x, D_x) v(t,x) + \phi(t,x), & (t,x) \in [0, \tau] \times \Omega, \\ \\
D_t v(t,x') + B(t_0+t,x', D_x) v(t,x') - L(t_0+t)[v(t,\cdot)_{\Gamma}] (x')  = \psi(t,x'), & (t,x') \in [0, \tau] \times \partial \Omega, \\ \\
v(0,x) = v_0(x), & x \in \Omega,
\end{array}
\right. 
\end{equation}
with $\tau \in (0, T-t_0]$. We have
$$
A(t_0+t, \cdot, D_x) = A(t_0,\cdot,D_x) + \A(t), 
$$
with $\A(t) = A(t_0+t, \cdot, D_x) - A(t_0,\cdot,D_x) $, 
$$
 B(t_0+t,\cdot, D_x) v(t,\cdot) - L(t_0+t)[v(t,\cdot)_{\Gamma}] = \mB(t)v(t,\cdot) - L(t_0)[v(t,\cdot)_{\Gamma}], 
$$
with
$$
\mB(t)v(t,\cdot) = B(t_0+t,\cdot, D_x) v(t,\cdot) - [L(t_0+t) - L(t_0)][v(t,\cdot)_{\Gamma}]. 
$$
By Remark \ref{re3.5},  there exists $M \in \R^+$ such that, $\forall t_0 \in [0, T)$,  $\forall \delta \in \R^+$, there exists $\epsilon(\delta) \in \R^+$ such that, if $|t - t_0| \leq \epsilon(\delta)$, $\forall u \in C^{2+\beta}(\Omega)$, (\ref{eq3.6B})-
(\ref{eq3.11B}) are fulfilled. We deduce from Lemmata {\ref{le3.4} and \ref{le3.8} and Remark \ref{re3.6} that there exists $\tau_0 \in \R^+$, independent of $t_0$, such that, if $\tau \leq \tau_0 \wedge (T-t_0)$, if $v_0 \in C^{2+\beta}(\Omega)$, $\phi \in 
C^{\beta/2,\beta}((0, \tau) \times \Omega)$, $\psi \in C^{\beta/2,\beta}((0, \tau) \times \Gamma)$ and 
$$
A(t_0,x', D_x) v_0(x') + \phi(0,x') = -B(t_0,x', D_x) v_0(x') - L(t_0)[v_{0|\Gamma}] (x')  + \psi(0,x'),
$$
then (\ref{eq3.21}) has a unique solution $v$ in $C^{1+\beta/2,2+\beta}((0, \tau) \times \Omega)$. So the proof of the existence and uniqueness of a solution $u$ to (\ref{eq1.1})   in $C^{1+\beta/2,2+\beta}((0, T) \times \Omega)$ can be
 performed with the same arguments in $(\delta)-(\epsilon)$ in the proof of Lemma \ref{le3.4}. 

$\square$

\medskip

{\bf Proof of Theorem \ref{th1.2}. } Clearly, (\ref{eq1.1B}) is equivalent to 
\begin{equation}\label{eq3.28}
\left\{\begin{array}{ll}
D_tu(t,x) = A(t,x, D_x) u(t,x) + f(t,x), & (t,x) \in [0, T] \times \Omega, \\ \\
D_tu(t,x') + B(t,x', D_x) u(t,x') - L(t)(u(t,\cdot)_{|\Gamma})(x')  = h(t,x') + f(t,x'), & (t,x') \in [0, T] \times \partial \Omega, \\ \\
u(0,x) = u_0(x), & x \in \Omega. 
\end{array}
\right. 
\end{equation}
So the conclusion follows immediately from Theorem \ref{th1.1}.

$\square$

\end{document}